\documentclass[reqno,11pt]{amsart}
\usepackage{amssymb}
\usepackage{amsbsy}
 \usepackage{tikz}
\usepackage{epsfig}
\usepackage{multicol}

\begin{document}
\pagestyle{myheadings}
\title{A conjecture of eigenvalues of threshold graphs}
\author{ Fernando Tura}
\address{ Departamento de Matem\'atica, UFSM,97105--900 Santa Maria, RS, Brazil}
\email{\tt ftura@smail.ufsm.br}

\pdfpagewidth 8.5 in
\pdfpageheight 11 in

\def\floor#1{\left\lfloor{#1}\right\rfloor}

\newcommand{\diagonal}[8]{
\begin{array}{| c | r |}
b_i & d_i \\
\hline
 #1 & #5 \\
        & \\
 #2 & #6 \\
        & \\
 #3  & #7 \\
         &  \\
 #4  & #8 \\
 \hline
 \end{array}
}

\newcommand{\thresholdmatrix}[8]{
\left [
     \begin{array}{ccccccc}
            &   &   &   &  #1   &   & #2 \\
            &   &   &   &  \vdots   &   & \vdots \\
            &   &   &   &  #1  &   & #2 \\
     \\
         #3  & \ldots   & #3 &   &  #4    &   & #5 \\
     \\
         #6  & \ldots   & #6  &   &  #7   &   &  #8 \\
    \end{array}
\right ]
}

\newcommand{\lambdamin}{\lambda_{\min, n}}
\newcommand{\formulamin}{
 \frac{  (
   \lfloor \frac{n}{3} \rfloor \! - \! 1 ) \! - \! \sqrt{ ( \lfloor \frac{n}{3} \rfloor \! - \! 1)^2 \! +\!  4
   (n \! - \! \lfloor \frac{n}{3} \rfloor )
   \lfloor \frac{n}{3} \rfloor
  }
  }{2}
}
\newcommand{\casei}{{\bf case~1}}
\newcommand{\subia}{{\bf subcase~1a}}
\newcommand{\subib}{{\bf subcase~1b}}
\newcommand{\subic}{{\bf subcase~1c}}
\newcommand{\caseii}{{\bf case~2}}
\newcommand{\subiia}{{\bf subcase~2a}}
\newcommand{\subiib}{{\bf subcase~2b}}
\newcommand{\caseiii}{{\bf case~3}}
\newcommand{\myvar}{x}
\newcommand{\exvar}{\frac{\sqrt{3} + 1}{2}}
\newcommand{\PrfSketch}{{\bf Proof (Sketch): }}
\newcommand{\boldQ}{\mbox{\bf Q}}
\newcommand{\boldR}{\mbox{\bf R}}
\newcommand{\boldZ}{\mbox{\bf Z}}
\newcommand{\boldc}{\mbox{\bf c}}
\newcommand{\sign}{\mbox{sign}}
\newcommand{\alphaseq}{{\pmb \alpha}_{G,\myvar}}
\newcommand{\alphaseqlambda}{{\pmb \alpha}_{G,\lambda}}
\newcommand{\alphaseqGprime}{{\pmb \alpha}_{G^\prime,\myvar}}
\newcommand{\alphaseqlam}{{\pmb \alpha}_{G,-\lambdamin}}
\newtheorem{Thr}{Theorem}
\newtheorem{Pro}{Proposition}
\newtheorem{Que}{Question}
\newtheorem{Con}{Conjecture}
\newtheorem{Cor}{Corollary}
\newtheorem{Lem}{Lemma}
\newtheorem{Fac}{Fact}
\newtheorem{Ex}{Example}
\newtheorem{Def}{Definition}
\newtheorem{Prop}{Proposition}
\def\floor#1{\left\lfloor{#1}\right\rfloor}

\newenvironment{my_enumerate}{
\begin{enumerate}
  \setlength{\baselineskip}{14pt}
  \setlength{\parskip}{0pt}
  \setlength{\parsep}{0pt}}{\end{enumerate}
}
\newenvironment{my_description}{
\begin{description}
  \setlength{\baselineskip}{14pt}
  \setlength{\parskip}{0pt}
  \setlength{\parsep}{0pt}}{\end{description}
}

\maketitle

\begin{abstract}
Let $A_n$ be the anti-regular graph of order $n.$ It was conjectured that among all threshold graphs on $n$ vertices, $A_n$ has  the smallest positive eigenvalue and the largest eigenvalue less than $-1.$
Recently, in \cite{Cesar2} was given partial results for this conjecture and identified the critical cases where a more refined method is needed. In this paper, we deal with these cases and confirm that conjecture holds.
\\

\noindent
{\bf keywords:} threshold graph, adjacency  matrix, eigenvalues.  
\end{abstract}

\section{Introduction}
\label{intro}
A simple graph $G=(V,E)$ is a threshold graph if there exists a function $w: V(G) \longrightarrow [0, \infty)$ 
and a real number $t\geq 0$ called the threshold such that $uv \in E(G)$ if and only if $w(u) + w(v) \geq t.$
This class of graphs was
introduced by Chv\'{a}tal and Hammer \cite{CH1977}
and Henderson and Zalcstein \cite{HZ1977} in 1977.
They are an important class of graphs because of their
numerous applications in many areas such as computer science  and psychology \cite{Mah95}.

One way to characterize threshold graphs is
through an iterative process which
starts with an isolated vertex, and where, at each step, either a new
{\em isolated} vertex is added, or a {\em dominating} vertex is added.
We represent a threshold graph $G$ on $n$ vertices using a
binary string $(b_1, \ldots, b_n)$.
Here $b_i = 0$ if vertex $v_i$ was added as an
isolated vertex, and $b_i = 1$ if $v_i$ was added as a dominating vertex.
We call our representation a {\em creation sequence},
and always take $b_1$ to be zero.
If $n \geq 2$, $G$ is connected if and only if $b_n = 1$.


There is a considerable body of knowledge on the spectral properties of threshold graphs. For example, all eigenvalues except $-1$ and $0$ are {\em main}, meaning that the entries in the associated eigenvector do not sum to zero (see \cite{SF2011}, Theorem 7.5).
With the exception of $-1$ and $0,$ all eigenvalues of threshold graphs are simple \cite{JTT2013}. 
In \cite{JTT2015} was proved that no threshold graph has eigenvalues in the interval $(-1,0).$
For more spectral properties we suggest consulting the articles \cite{Cesar2, Cesar, Bapat, gorbani, JOT2019, lou}.

A distinguished subclass of threshold graphs is the family of anti-regular graphs $A_n$ which are the graphs with only two vertices of equal degrees.
The Figure \ref{figure1} shows the graph $A_{16}.$
In \cite{Cesar}, a nearly complete characterization of the eigenvalues of anti-regular graphs is given, and
was proposed some conjectures about it. Among them, we consider in this work the following one:

\begin{Con} 
\label{conjec}
For each $n,$ the anti-regular graph  $A_n$  has the smallest positive eigenvalue and the largest negative eigenvalue less than $-1$  among all threshold graphs 
on $n$ vertices.
\end{Con}
Recently in \cite{Cesar2}, was given partial results for this conjecture and identified the critical cases where a more refined method is needed. More exactly, the conjecture was proved for all threshold graphs on $n$ vertices except for $n-2$ critical cases where the interlacing method fails. 
In this paper, we deal with these cases and confirms that conjecture holds.

The paper is organized as follows. The main tool used to prove the conjecture, and some known results are reviewed in Section~\ref{Diag}. In Section~\ref{main}, we present some auxiliaries results and finally in Section~\ref{conjecture} we confirm that conjecture holds for the remaining cases.

\begin{figure}[h!]
       \epsfig{file=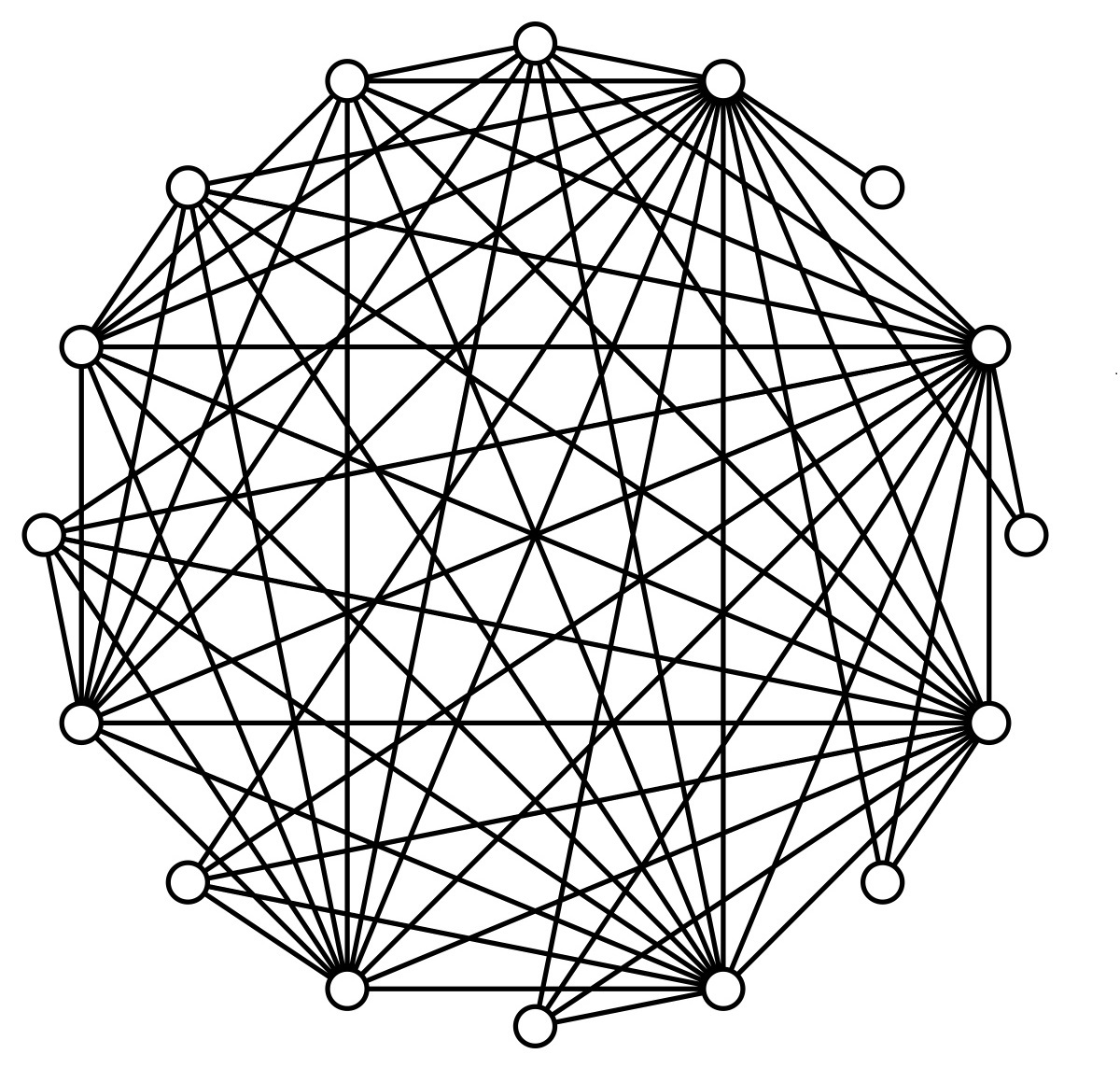, height=5.5cm, width=7cm}
\caption{ The anti-regular graph $A_{16}$}
\label{figure1}
       \end{figure}


\section{Background Results}
\label{Diag}
Recall that two matrices $R$ and $S$ are {\em congruent}
if there exists a nonsingular matrix $P$ such that $R = P^T S P$.
An important tool used in \cite{JTT2013} was an algorithm
for constructing a {\em diagonal} matrix $D$ congruent to $A + \myvar I$,
where $A$ is the adjacency matrix of a threshold graph,
and $\myvar$ is an arbitrary scalar.
The algorithm is shown in Figure~\ref{algo}.
The diagonal elements are stored in the array $d$,
and the graph's initial representation is stored in $b$.

\begin{figure}[h]
{\tt
\begin{tabbing}
aaa\=aaa\=aaa\=aaa\=aaa\=aaa\=aaa\=aaa\= \kill
     \>  Algorithm $\mbox{Diagonalize}(G, \myvar)$ \\
     \> \> initialize $d_i  \leftarrow \myvar$, for all $i$ \\
     \> \> {\bf for } $m = n$ to $2$ \\
     \> \> \>     $\alpha \leftarrow d_m$ \\
     \> \> \> {\bf if} $b_{m-1} = 1$ and $b_{m} = 1$ \\
     \> \> \> \>  {\bf if} $\alpha + \myvar \neq  2$  \verb+            //subcase 1a+ \\
     \> \> \> \> \>   $d_{m-1} \leftarrow  \frac{\alpha \myvar - 1}{\alpha + \myvar - 2}$ \\
     \> \> \> \> \>   $d_m \leftarrow  \alpha + \myvar - 2$  \\
     \> \> \> \>  {\bf else if} $\myvar =  1$  \verb+           //subcase 1b+  \\
     \> \> \> \> \>   $d_{m-1} \leftarrow 1$  \\
     \> \> \> \> \>   $d_m  \leftarrow 0$  \\
     \> \> \> \>  {\bf else }                   \verb+                  //subcase 1c+ \\
     \> \> \> \> \>   $d_{m-1} \leftarrow 1$  \\
     \> \> \> \> \>   $d_m \leftarrow - (1 - \myvar)^2$  \\
     \> \> \> \> \>   $b_{m-1} \leftarrow 0 $  \\
     \> \> \> {\bf else if} $b_{m-1} = 0$ and $b_{m} = 1$ \\
     \> \> \> \>  {\bf if} $\myvar = 0$  \verb+               //subcase 2a+    \\
     \> \> \> \> \>   $d_{m-1} \leftarrow 1$ \\
     \> \> \> \> \>   $d_m \leftarrow -1$ \\
     \> \> \> \>  {\bf else }  \verb+                  //subcase 2b+   \\
     \> \> \> \> \>   $d_{m-1} \leftarrow \alpha - \frac{1}{\myvar}$ \\
     \> \> \> \> \>   $d_m  \leftarrow \myvar$ \\
     \> \> \> \> \>   $b_{m-1} \leftarrow 1 $  \\
     \> \>  {\bf end loop}
\end{tabbing}
}
 \caption{\label{algo} Algorithm Diagonalize.}
\end{figure}

Algorithm \verb+Diagonalize+ works bottom up.
For a graph of order $n$, it makes $n-1$ passes.
Each diagonal element,
except the first and last,
participates in two iterations.
During each iteration, the assignment
to $d_m$ produces a final diagonal element.
On the last iteration, when $m = 2$, the
assignment to $d_{m-1}$ also produces
a final diagonal element at the top.

Note when $b_m = 0$, the algorithm does nothing and moves to the next step.
Also note that the values in $b$ can change.
In each iteration, the algorithm executes one of the five subcases.
It should be noted that \subia~ and \subiib~ are the normal cases,
and the other three subcases represent singularities.
Executing \subib~ requires $\myvar = 1$,
executing \subiia~ requires $\myvar = 0$,
and executing \subic~ requires $\alpha + \myvar = 2$.

The next  result from \cite{JTT2013} will be used throughout the paper.

\begin{Thr}
\label{main1}
Let $G$ be a threshold graph and let $(d_v)_{v \in G}$ be the sequence produced by Diagonalize $(G,-x).$
Then the diagonal matrix $D = diag(d_v)_{v \in G}$ is congruent to $A(G) -xI,$ so that  the number of (positive - negative - zero) entries in $(d_v)_{v\in G}$ 
is equal to the number eigenvalues of $A(G)$  that are (greater than $\myvar$ - small than $\myvar$ - equal to $\myvar$).  
\end{Thr}

\begin{Lem}
\label{1c}
If algorithm \verb+Diagonalize+ executes  \subic, then it will leave both a permanent negative and positive number on the diagonal.
\end{Lem}
\begin{proof}
The assignment $d_m \leftarrow -(1-x)^2$ produces a negative number. The positive number written occurs with $d_{m-1} \leftarrow 1.$ Normally, assignments to $d_{m-1}$ are overwritten in the next iteration. However, since $b_{m-1} \leftarrow 0,$ the following iteration will leave this entry unchanged.
\end{proof}


Given a graph $G$, we let $n_{+}(G)$ and $n_{-}(G)$
denote respectively the number of positive and negative eigenvalues of $G$,
and $n_{0}(G)$ and $n_{-1}(G)$ denote the multiplicities of $0$ and $-1$.
The triple $( n_{+}(G), n_{0}(G), n_{-}(G) )$
is called the {\em inertia} of $G$.

The following result is due to Bapat \cite{Bapat}.
\begin{Thr}
\label{numberones}
In a connected threshold graph $G$ represented with
$\mathbf{b}$,
$n_{-}(G)$ is the number of $1$'s in
$\mathbf{b}$, 
and $n_{0}(G)$ is the number of substrings $00$ in $\mathbf{b}$.
\end{Thr}

Note that in the creation sequence of a connected threshold graph,
every zero must be followed by a zero or one.
So $u$ the number of zeros,
equals $u_{00}$ the number of substrings $00$,
plus $u_{01}$ the number of substrings $01$.
If $v$ is the number of ones in the sequence,
$n = v + u = v + u_{00} + u_{01}$.
Therefore,
$n_{+}(G) = n - n_{-}(G) - n_{0}(G) = n - v - u_{00} = u_{01}$.
That is,
\begin{Thr}
\label{numberzeroone}
In a connected threshold graph $G$,
the number of occurrences of the substring $01$
in its creation sequence equals $n_{+}(G),$ and $n_{-1}(G)$ is the number of substrings $11$ in $\mathbf{b}$. 
\end{Thr}


\section{Basic Results}
\label{main}

Throughout  this section we let $G$ be  a connected threshold graph of order $n \geq 3,$ whose $\lambda(G)$ is a simple eigenvalue  $\lambda(G) \neq -1,0.$

\begin{Lem}
\label{lema2}
If $G$ is a threshold graph on $n$ vertices with  $x= -\lambda(G) ,$ then \verb+Diagonalize+$(G, x)$ 
produces a zero at the top of the diagonal.
\label{alphaleqone}
\end{Lem}
\begin{proof}
Since $-x$ is an eigenvalue, by Theorem \ref{main1} we must obtain a zero on the diagonal. 
  An inspection of the algorithm shows that since $x \neq 0,1,$ a zero can be written only during the algorithm's last iteration, to the top of the diagonal.
\end{proof}

\begin{Lem}
\label{lema3}
Let $G$ and $H$ be two  threshold graphs on $n$ vertices with  their respective eigenvalues $\lambda(G),\lambda(H)\neq -1,0.$
\begin{my_description}
\item[i]
If the number of negative entries in \verb+Diagonalize+$(H, x)$ exceeds the number of negative entries in \verb+Diagonalize+$(G, x)$ by one, where $x=-\lambda(G),$
and $\lambda(G), \lambda(H)< -1$  then $\lambda(H) < \lambda(G)$ 
\item[ii]If the number of  positive entries in \verb+Diagonalize+$(H, x)$ exceeds the number of positive entries in \verb+Diagonalize+$(G, x)$ by one, where $x=-\lambda(G),$
and $0< \lambda(G), \lambda(H)$  then $\lambda(G) < \lambda(H).$ 
\end{my_description}
\end{Lem}

\begin{proof} We check  item $(i).$
Since that the number of negative entries in \verb+Diagonalize+$(G, x),$ where $x=-\lambda(G)$ corresponds to the number of eigenvalues of $G$ that are small than $\lambda(G),$
by Theorem \ref{main1}, and  \verb+Diagonalize+$(H, x)$ exceeds this number by one, then the largest eigenvalue
of $H$ less than $-1$ is smaller than $\lambda(G),$ that is, $\lambda(H) < \lambda(G).$ The item $(ii)$ is similar.
\end{proof}


During diagonalization $(G, x),$ where $x = -\lambda(G),$ it is impossible for the algorithm
to enter \subib~ or  \subiia~    because $x \neq 0,1.$    
As we will show in the next section
\verb+Diagonalize+$(G, x)$ does not execute the \subic. Then it must enter \subia~ or \subiib~ initially, and remain in one of these two subcases.
The key to solve our problem is to understand the behavior of the following functions:
\begin{equation}
\label{g}
 g(\alpha)= \frac{\alpha x -1}{\alpha + x -2}
\end{equation}
\begin{equation}
\label{f}
 f(\alpha)=  \alpha - \frac{1}{x}
\end{equation}
These functions, of course, are used in  \subia~ and \subiib, respectively.
We regard $x$ as fixed and $\alpha$ is an indetermined.

During the execution  of \verb+ Diagonalize+ $(G,x),$ there is a  sequence  of $n$ values  calculated right to left
\begin{equation}
{\bf \alpha_{G,x}} = (\alpha_1, \alpha_2, \ldots
\alpha_{n-1},\alpha_{n}=x)
\end{equation}
that are  temporarily  assigned to the diagonal, we call the  $\alpha$-sequence.  Except for the  final value
$\alpha_1,$  each gets overwritten.  They are computed:
$$ \alpha_{i-1} = h_{i}(\alpha_i),\hspace{0,5cm} 2 \leq i \leq n$$
where
$$h_{i}(\alpha) =  \left\{\begin{array}{ccc}
g(\alpha) ,& \mbox{ se}  \hspace{0,25cm}  {b_{i-1} =1}  &\\
f(\alpha)  ,& \mbox{ se}  \hspace{0,25cm}  {b_{i-1} =0}  \\
\end{array}\right.$$
and $g$ and $f$ are defined in (\ref{g}) and (\ref{f}). As compositions we have 

\begin{equation}
\label{cinco}
\alpha_1 = h_{2}\circ h_{3}\circ \ldots \circ h_{i}(\alpha_i)
\end{equation}

\noindent  for $2\leq i \leq n.$ 
The sequence of $h_i $ depends only on the original $b_i.$

The Figure \ref{figure3} illustrates the functions $f$ and $g$ for $x=4.$

\begin{figure}[h!]
       \epsfig{file=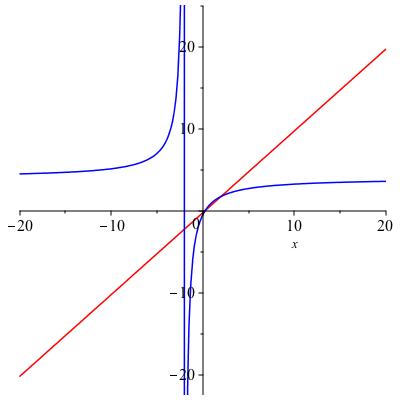, height=6cm, width=7cm}
\caption{ The functions $f$ and $g$}
\label{figure3}
       \end{figure}

\begin{Lem}
\label{lema4}
Both  $f$ and $g$ are  continuous and increasing on $(2-x, +\infty) .$
\end{Lem}
\begin{proof}
 Their derivatives $\frac{df}{d\alpha}= 1$ and
$\frac{dg}{d\alpha}= \frac{(x-1)^2}{(\alpha +x -2)^2}$  are
positive.
\end{proof}

\begin{Lem}
\label{lema5}
The following properties hold for $f$ and $g:$
\begin{my_description}
\item[i]
$f(\alpha) = 0$ if and only if $\alpha = \frac{1}{x}.$ 
\item[ii]
If $  \alpha < \frac{1}{x} , x>0$  and $\alpha +x -2 <0$ then $ f(\alpha)< 0$ and $g(\alpha) >0.$
\item[iii]
If  $0<  \frac{1}{x} < \alpha < 2 $ and $\alpha +x -2 > 0$ then $ f(\alpha) < g(\alpha).$
\item[iv]
If  $ \alpha > 2 $ then $ g(\alpha) < f(\alpha).$
\end{my_description}
\end{Lem}
\begin{proof}
Properties {\bf i} and {\bf ii} are easily verified. To see {\bf iii}, we note that
 $f(\alpha) < g(\alpha) $ we must show that
$$ \alpha - \frac{1}{x} = \frac{\alpha x -1 }{x} < \frac{\alpha x -1}{\alpha + x -2},$$
which is equivalent to $\frac{1}{x} < \frac{1}{\alpha +x -2},$ which holds since $\alpha +x -2 < x.$ 
Using similar argument we prove the item {\bf iv.}
\end{proof}

\begin{Lem}
\label{lema6} Let $H$ be a connected threshold graph obtained from $G$ by changing a single $b_l, 1 < l < n,$ from $1$ to $0,$ and consider the execution of \verb+Diagonalize+$(H,x).$ 
\begin{my_description}
\item[i]
If $ \frac{1}{x} < \alpha_{l+1} < 2,$ for $x>0,$ and $h_{l+1}= g$  is replaced by $f,$ then    $\alpha'_1$  will decrease.
\item[ii] If $  \alpha_{l+1} < \frac{1}{x},$ for $x>0,$ and $h_{l+1}= g$ is replaced by $f,$ then $\alpha'_1$ will decrease. 
\item[iii] If $  \frac{1}{x} < \alpha_{l+1} ,$ for $x<0,$ and $h_{l+1}= g$ is replaced by $f,$ then $\alpha'_1$ will increase. 

\end{my_description}

\end{Lem}
\begin{proof}
We prove item $(i).$ Assuming \subic~ is avoided, the new alpha sequence
\begin{equation}
\label{seven}
\alpha_{H,x} = (\alpha'_1, \ldots,\alpha'_l, \alpha_{l+1},
\ldots, \alpha_k = x)
\end{equation}
is computed  exactly  the same, except  $h_{l+1}$ will
change from  $g$  to $f.$ Since $  \frac{1}{x} < \alpha_{l+1} < 2,$ by Lemma \ref{lema5} (part {\bf iii})
$ \alpha_{l} = f(\alpha_{l+1}) <  g(\alpha_{l+1}) .$
Let  $h = h_2 \circ h_3 \circ \ldots \circ h_l $ be the remaining composition in (\ref{cinco}). By Lemma \ref{lema4}, each $h_i$ is continuous and increasing on   $( 2-x,+\infty) ,$ so the composition must be continuous and increasing, and we have: $ \alpha'_{1} = h(\alpha'_{k}) <  h(\alpha_{k}) = \alpha_{1}.$ The proof is similar for items $(ii)$ and $(iii).$
\end{proof}


Analogously, the following result can be verified.

\begin{Lem}
\label{lema7}
Let $H$ denote the connected threshold graph obtained from $G$ by changing a single $b_l, 1 < l < k,$ from $0$ to $1.$ Consider the execution of \verb+Diagonalize+$(H,x),$
\begin{my_description}
\item[i]  If $ \alpha_{l+1} > 2,$ for $x>0$ and $h_{l+1}= f$  is replaced by $g,$ then    $\alpha'_1$  will decrease.
  
\item[ii] If $ \alpha_{l+1}< \frac{1}{x},$ for $x<0$ and $h_{l+1}= f$  is replaced by $g,$ then    $\alpha'_1$  will increase. 
\end{my_description}

\end{Lem}

\section{The proof of Conjecture}
\label{conjecture}

The connected anti-regular graph on $n$ vertices, denoted by $A_n$ is a threshold graph 
with binary sequence $b=(0101 \ldots 01)$ when $n$ is even and $b=(00101\ldots 01)$ when $n$ is odd. It was proved in \cite{Cesar} (see, also \cite{JTT2013}) that $A_n$ has simple eigenvalues and moreover has inertia $i(A_{2k})=(k, 0, k)$ if $n=2k$ is even and $i(A_{2k+1})=(k,1,k)$ if $n=2k+1$ is odd, and therefore $\lambda_{k+1}(A_{2k+1})=0$ and $\lambda_k(A_{2k})=-1.$

The following result is due \cite{Cesar2}.
\begin{Prop}
The interval $\Omega= [\frac{-1-\sqrt{2}}{2}, \frac{-1+\sqrt{2}}{2}]$ does not contain any eigenvalue $\lambda\neq -1,0$ of any threshold graph.
\end{Prop}

Now, let introduce the $n-2$ critical threshold graphs. According cited in \cite{Cesar2} they are identified having binary sequence $G=(0^{s_1} 1^{t_1} \ldots 0^{s_k} 1^{t_k})$ such that 
\begin{itemize}
\item If $n=2k+2$ then either $s_1=2$ and exactly one of $s_2, \ldots, s_k, t_1, \ldots, t_k$ is also equal to two and all others are one, or $s_1=3$ and all other $s_i=t_i=1.$
\item If $n=2k+1$ then either $s_1=1$ and only  one of $s_2, \ldots, s_k, t_1, \ldots, t_k$ equals two and all others equal one.
\end{itemize}

Recall that $\lambda^{-}(G)$ denotes the largest eigenvalue of  a critical threshold graph $G$ less than $-1$ and $\lambda^{+}(G)$ denotes the smallest positive eigenvalue of $G.$
For completing the proof of Conjecture \ref{conjec}, we need to show that $\lambda^{-}(G) \leq \lambda^{-}(A_n)$ if $n$ is even, and 
$\lambda^{+}(A_n) \leq \lambda^{+}(G)$ if $n$ is odd.

\subsection{ The case $n$ is even}

\begin{Lem}
\label{lemamax}
Let $G$ be a graph having the  largest eigenvalue $\lambda^{-}(G)$ among all  $n-2$ critical threshold graphs with $s_1=2.$  
\begin{my_description}
\item[i]
If $G$ has binary sequence $G=(b_1, b_2, \ldots, b_{i-1}, 0,0,1,0,1,\ldots,0,1)$ then after processing $b_{i+1}=0$ and $b_{i+2}=1$ by  \verb+Diagonalize+$(G,x)$
the assignment   is $\alpha < \frac{1}{x},$ where $x=-\lambda^{-}(G).$
\item[ii]
If $G$ has binary sequence $G=(b_1, b_2, \ldots, b_{i-1}, 1,1,0, \ldots,0,1)$  then after processing $b_{i+1}=1$ and $b_{i+2}=0$ by \verb+Diagonalize+$(G,x)$
has assignment $\frac{1}{x} < \alpha,$ where $x=-\lambda^{-}(G).$
\end{my_description}
\end{Lem}
\begin{proof} First we note that during execution of \verb+Diagonalize+$(G,x)$ where $x=-\lambda^{-}(G),$ we must have $\alpha=\frac{1}{x}$ only in the step $m=2,$
according to Lemma \ref{lema2} and Lemma \ref{lema5} (part $(i)).$
Now, we check the  item $(i).$ 

Let $G$ having binary sequence $G=(b_1, b_2, \ldots, b_{i-1}, 0,0,1,0,1,\ldots,0,1).$
 We assume that in the $(i+1)-$th iteration of \verb+Diagonalize+$(G,x)$
has assigned $\alpha > \frac{1}{x}.$
Let $H$ be the threshold graph obtained from $G$  by changing a single $b_{i+1}$ from $0$ to $1,$
and consider \verb+Diagonalize+$(H,x),$ where $x=-\lambda^{-}(G).$ It is easy to see that in $i-$th iteration of \verb+Diagonalize+$(G,x)$ we will have $f(\alpha) >0$ while that in \verb+Diagonalize+$(H,x)$  we will have $g(\alpha)>0,$ such that $0< f(\alpha) <g(\alpha).$ Taking into account remaining elements of binary sequence are equal follows  \verb+Diagonalize+$(H,x)$ will assigned $\alpha'_1 > \alpha_1=0.$  Since $G$ and $H$ have the same number of substrings $01$ (and therefore have the same number of positive eigenvalues) follows $\lambda^{-}(G) < \lambda^{-}(H),$ what is a contradiction.
 The proof for item $(ii)$ is similar.
\end{proof}

Let $G$ be a threshold graph having binary sequence $G=(0^{s_1} 1^{t_1} \ldots 0^{s_k} 1^{t_k})$ and let $\lambda^{-}(G)$ be the largest eigenvalue less than $-1.$
Let $\delta_{d}(G)$ denotes the signal of the final diagonal of  \verb+Diagonalize+$(G,x),$ where $x=-\lambda^{-}(G).$

It follows from the Lemma \ref{lemamax} the following result.
\begin{Thr}
\label{mainan}
Let $A_n$ be the anti-regular graph on $n$ vertices and let $(d_i)$ be the final diagonal of
\verb+Diagonalize+$(A_n,y).$
If $n\geq 3$  then
$$
\delta_{d}(A_n)= \left\{ \begin{array}{rl}
 (+, +, -,\ldots,-,+)  &\mbox{ if $n$ is even and $-y \in (\lambda^{-}(A_n),-1)$} \\
 (-,+,+,\ldots,-,+) &\mbox{ if $n$ is odd and $-y \in (\lambda^{-}(A_n),0)$}
       \end{array} \right.
$$
\end{Thr}

\noindent{\bf Remark:} Let $G$ be one of $n-2$ critical threshold graphs. Note that during execution of \verb+Diagonalize+$(G, x),$ with $x=-\lambda^{-}(G)$ the \subic~cannot occur
for $m=2$ nor for $m=3.$ If it occurs for an intermediate step then implies each substring of type $(1010)$ has final sign equal to  $(+,+,+,-)$ contrary to Theorem \ref{mainan}.

For showing the conjecture, we first need to prove the following result.

\begin{Thr}
\label{main5}
Let $G$ be one of $n-2$ critical threshold graphs on $n=2k$ vertices. Then holds:
\begin{my_description}
\item[i] $\lambda^{-}(0^{2} 1 01 \ldots 0^{2} 1 01 \ldots 01) < \lambda^{-}(0^{2} 1 01 \ldots 0 1^{2} 01 \ldots 01)$ for each $(s_k , t_{k})$ and $\lambda^{-}(0^{2} 1 01 \ldots 0 1^{2} 01 \ldots 01) < \lambda^{-}(0^{2} 1 01 \ldots 0 10^{2} 1 \ldots 01)$ for  each $(t_k,s_{k+1})$  and $k>\frac{n}{2}.$
\item[ii]  $\lambda^{-} (0^{2} 1 01 \ldots 0 1^{2}  01 \ldots 01) <\lambda^{-} (0^{2} 1 01 \ldots 0^2 10 \ldots 01)$  for each $(s_{k},t_k)$ and $\lambda^{-} (0^{2} 1 01 \ldots 0^2 101 \ldots 01) < \lambda^{-} (0^{2} 1 01 \ldots 0 1^{2}  01 \ldots 01)$ for each $(t_{k-1}, s_k),$ and $k\leq \frac{n}{2}.$
\end{my_description}

\end{Thr}

\begin{proof} Let denotes by $G_1= (0^{2} 1 01 \ldots 0^{2} 1 01 \ldots 01)$ with $s_k=2$ and $G_2= (0^{2} 1 01 \ldots 0 1^{2} 01 \ldots 01)$ with $t_{k}=2$ for $k>\frac{n}{2}.$
 We  show the inequality $\lambda^{-}(G_1)< \lambda^{-}(G_2).$ According to Lemma \ref{lema3} it is suffices to show that number
of negative entries  in \verb+Diagonalize+$(G_1,x),$ exceed  by one the number of negative entries in \verb+Diagonalize+$(G_2,x),$ where $x=-\lambda(G_2).$  

We consider the \verb+Diagonalize+$(G_2,x),$ where $x=-\lambda^{-}(G_2).$
Since that entries positive corresponds to the number of substrings $01,$ $11$ and $00$ in the creation sequence, and using the Lemma \ref{lema2} 
\begin{equation}
\label{eq6}
\delta_{d}(G_2) = (0,+,+,-,+,\ldots,-,+,+, -,+, \ldots,-,+)
\end{equation}

 Now, we consider the \verb+Diagonalize+$(G_1,x),$ where $x=-\lambda(G_2).$ 
Since $G_1$ and  $G_2$ have the same sequence $b_i$ for $i>t_k$ then they have the same values and therefore the same  signs.
Let $\alpha$ be the most recent assignment common to both graphs. If $G_2$ is the graph with largest $\lambda^{-}(G)$ then by Lemma \ref{lemamax}
 we have $\frac{1}{x} < \alpha$ and $\alpha <2,$ since that \subia~was executed. 
Furthermore $G_1$ can be obtained from $G_2$ by changing a single $b_i$ from $1$ to $0.$ By Lemma \ref{lema6} (item $i$)  we will have the final value $ \alpha'_1 < \alpha_1=0,$ that is 
\begin{equation}
\label{eq7}
\delta_{d}(G_1) = (-,+,+,-,+,\ldots,-,+,+,-,+, \ldots,-,+)
\end{equation}

Therefore thus  comparing the signs of final diagonal of both graphs in (\ref{eq6}) and (\ref{eq7}) we have  $\lambda^{-}(G_1) < \lambda^{-}(G_2).$
The proof is similar for the others items.
\end{proof}


\begin{Cor}
\label{cor1}
Among all threshold graphs of order $n=2k,$ the anti-regular graph $A_n$ has  the largest eigenvalue less than $-1.$
\end{Cor}

\begin{proof}
Let $G_1, G_2, G_3$ and $A_n$ be  threshold graphs having binary sequence $G_1= (0^{2} 1^2 01 \ldots 0 1  01 \ldots 01),$  
$G_2= (0^{3} 1 01 \ldots 0 10 01 \ldots 01),$  $G_3= (0^{2} 1 01 \ldots \\0 1  01 \ldots 01^2)$ and $A_n=(010101\ldots 0101).$  
 We claim that
 \begin{equation}
 \label{eq8}
  \lambda^{-}(G_1) < \lambda^{-}(G_2) < \lambda^{-}(A_n)
 \end{equation} 
 and 
  \begin{equation}
  \label{eq9}
  \lambda^{-}(G_3) <  \lambda^{-}(A_n)
 \end{equation} 
  The inequality on the left in (\ref{eq8}) is proved by similar way to  Theorem \ref{main5} above.
  Now, we check the inequality on the right in (\ref{eq8}). 
  
  We consider the \verb+Diagonalize+$(G_2,x),$ where $x=-\lambda^{-}(A_n).$
Since $G_2$ and $A_n$ have the same $b_i$ for $i\geq n-3$  then
\begin{equation}
\label{eq1}
\delta_{d}(G_2) = (\delta(d_1),\delta(d_2),\delta(d_3),+,-,+,\ldots,-,+)
\end{equation}
 
 We will show that $d_1<0$ and $d_2, d_3 >0.$ Since $b_1=b_2=b_3=0$ and $b_4=1$ the \subiib~occurs for the last three steps of algorithm, then
  $d_2, d_3 >0.$ 
 To see $d_1<0,$ we consider the assignment  $\alpha$ of \verb+Diagonalize+$(A_n,x),$ where $m=3.$
 Note we must have in $m=2$ the assignment $\frac{1}{x}$ in \verb+Diagonalize+$(A_n,x),$ according Lemma \ref{lema5} (item $i$), and \subia~ was executed in the previous step, follows that $\alpha = \frac{2}{x+1}$ is the assignment in $m=3.$ Since $\frac{1}{x} < \alpha  <2$ 
 and, $G_2$ can be obtained from $A_n$ by changing  $b_2$ from $1$ to $0.$ 
 Therefore thus, by Lemma \ref{lema6} (item $i$) follows that $\alpha'_1 < \alpha_1=0.$ Finally,  comparing the signs of final diagonal of both graphs  we have  $\lambda^{-}(G_2) < \lambda^{-}(A_n).$
 
Now, let $(d_i)$ be the final diagonal of \verb+Diagonalize+$(G_3,x),$ with $x=-\lambda^{-}(A_n).$ We claim that
\begin{equation}
\label{eq2}
\delta_{d}(G_3) = (-,+,+,-,+,\ldots,-,+,+)
\end{equation}

The \subia~ occurs in the first step of \verb+Diagonalize+$(G_3,x).$ Since that $x>1$ we have that $d_n= 2(x-1) >0$ and $\alpha=\frac{x+1}{2}.$ Now, we have a subgraph isomorphic to anti-regular graph $A_{n-1}$ and assignment  $\frac{x+1}{2} < x =-\lambda^{-}(A_n).$  Follows each
substring $01,$ left a positive value and each substring $10,$ left a negative value. 
It remains to check only the sign of the last iteration. Since \subiib~occurs in the last iteration, we have that $d_2>0.$
We suppose that $d_1>0.$ It implies that $A_{n-1}$ has $\lfloor \frac{n-1}{2} \rfloor +2$ eigenvalues greater than $-(\frac{x+1}{2}).$ Since $A_{n-1}$ has exactly 
$\lfloor \frac{n-1}{2} \rfloor$ positive eigenvalues, and $0$
is a simple eigenvalue, follows that $A_{n-1}$ has an eigenvalue in the interval  $(\frac{-1-\sqrt{2}}{2}, 0),$ what is a contradiction.
Thus, we must have $d_1 <  0,$  and comparing the signs of final diagonal of both graphs we have that $\lambda^{-}(G_3) < \lambda^{-}(A_n)$ as desired.
\end{proof}

\subsection{ The case $n$ is odd}

We now treat the case $n$ odd. Using a procedure similar to Theorem \ref{main5} with Lemma \ref{lema6} (part iii) and Lemma \ref{lema7}(part ii), it can be verified that:
\begin{itemize}
\item $\lambda^{+}(01^{2} 01 \ldots 01)< \lambda^{+}(0 10^{2} 01 \ldots 01)< \ldots < \lambda^{+}(0 1\ldots 1 0^{2} \ldots 01)$ 
for $ s_{k}=2,$  $t_k=2$ and   $k \leq \lfloor \frac{n}{2} \rfloor -1$ and 

\item $\lambda^{+}(0101 \ldots  0101^2)< \lambda^{+}(0 1 01 \ldots   010^{2}1 )< \ldots < \lambda^{+}(0 1\ldots 1 0^{2} \ldots 01)$
for $s_{k}=2,$ $t_k=2$  and $k\geq  \lfloor \frac{n}{2} \rfloor -1$
\end{itemize}

\begin{Cor}
\label{cor2}
Among all threshold graphs of order $n=2k+1,$ the anti-regular graph $A_n$ has  the smallest positive eigenvalue.
\end{Cor}

\begin{proof}
It is sufficient to show the following inequalities
\begin{equation}
\label{eq12}
\lambda^{+}(00101 \ldots 01)< \lambda^{+}(01^201 \ldots 01)
\end{equation}
and
\begin{equation}
\label{eq13}
\lambda^{+}(00101 \ldots 01)< \lambda^{+}(0101 \ldots 01^2)
\end{equation}

Let denotes $G=(01^201\ldots 01)$ and $A_n=(00101\ldots 01).$
  We consider the \verb+Diagonalize+$(G,x),$ where $x=-\lambda^{+}(A_n).$
Since $G$ and $A_n$ have the same $b_i,$ for $i\geq3,$ in their creation sequence, then
\begin{equation}
\label{eq1}
\delta_{d}(G) = (\delta(d_1),\delta(d_2),\delta(d_3),+,-,+,-, \ldots,+,-)
\end{equation}
 
 We will show that $d_1>0$ and $d_2, d_3 <0.$ 
   We consider the assignment  $\alpha$ of \verb+Diagonalize+$(A_n,x),$ where $m=3.$
 Note we must have in $m=2$ the assignment $\frac{1}{x}$ in \verb+Diagonalize+$(A_n,x),$ according Lemma \ref{lema5} (item $i.$), follows that $\alpha = \frac{2}{x}$ is also the assignment in $m=3$ in in \verb+Diagonalize+$(G,x).$  Then \subia~ gives the following assigments: $d_3=\frac{2}{x}+x-2<0$ and $d_2=\frac{1}{2/x+x-2} \in(-1,0).$ Then the \subiib~occurs for the last iteration and gives:
 $d_2=x<0,$ and $d_1= \alpha -\frac{1}{x} >0,$ since that $\alpha \in (-1,0)$ and $x<0.$
  
  Finally,  comparing the signs of final diagonal of both graphs we have that $\lambda^{+}(A_n) < \lambda^{+}(G).$
For the inequality (\ref{eq13}) use the same procedure of second part of Corollary \ref{cor1}  by changing the signs of graphs.
\end{proof}

\end{document}